\documentclass[12pt]{amsart}

\usepackage[utf8]{inputenc}
\usepackage[T1]{fontenc}
\usepackage{amsmath,amssymb,amsthm,mathtools}
\usepackage[margin=1.1in]{geometry}
\usepackage[dvipsnames]{xcolor}
\usepackage{tikz-cd}
\definecolor{bodylinkblue}{RGB}{0,64,255}
\definecolor{toclinkred}{RGB}{255,0,0}
\definecolor{tocsubsectionred}{RGB}{255,0,0}
\usepackage[colorlinks=true,linkcolor=bodylinkblue,citecolor=bodylinkblue,urlcolor=bodylinkblue]{hyperref}

\numberwithin{equation}{section}
\allowdisplaybreaks

\newtheorem{theorem}{Theorem}[section]
\newtheorem{proposition}[theorem]{Proposition}
\newtheorem{lemma}[theorem]{Lemma}
\newtheorem{corollary}[theorem]{Corollary}

\newtheorem{problem}[theorem]{Problem}

\newtheorem{statement}[theorem]{Statement}
\theoremstyle{definition}
\newtheorem{definition}[theorem]{Definition}
\newtheorem{example}[theorem]{Example}
\theoremstyle{remark}
\newtheorem{remark}[theorem]{Remark}

\makeatletter
\def\l@section{\@tocline{1}{0pt}{1pc}{}{\color{toclinkred}}}
\def\l@subsection{\@tocline{2}{0pt}{1pc}{5pc}{\color{tocsubsectionred}}}
\def\l@subsubsection{\@tocline{3}{0pt}{1pc}{7pc}{\color{tocsubsectionred}}}
\renewcommand{\tocsection}[3]{%
  \textcolor{toclinkred}{\indentlabel{\@ifnotempty{#2}{\ignorespaces#1 #2.\quad}}#3}}
\renewcommand{\tocsubsection}[3]{%
  \textcolor{tocsubsectionred}{\indentlabel{\@ifnotempty{#2}{\ignorespaces#1 #2.\quad}}#3}}
\renewcommand{\tocsubsubsection}[3]{%
  \textcolor{tocsubsectionred}{\indentlabel{\@ifnotempty{#2}{\ignorespaces#1 #2.\quad}}#3}}
\makeatother

\newcommand{\im}{\sqrt{-1}}
\newcommand{\C}{\mathbb C}
\newcommand{\R}{\mathbb R}

\newcommand{\mT}{\mathbb T}
\newcommand{\CP}{\mathbb P}

\newcommand{\mO}{\mathcal O}
\newcommand{\w}{\omega}
\newcommand{\Ric}{\operatorname{Ric}}
\newcommand{\HSC}{\operatorname{HSC}}

\newcommand{\sff}{\mathrm{II}}
\newcommand{\td}{\operatorname{td}}

\newcommand{\Span}{\operatorname{Span}}
\newcommand{\tf}{\mathrm{tf}}

\newcommand{\Z}{\mathbb Z}
\newcommand{\Bl}{\operatorname{Bl}}
\newcommand{\Aut}{\operatorname{Aut}}
\newcommand{\Hor}{\mathcal H}
\newcommand{\Ver}{\mathcal V}
\newcommand{\calY}{\mathcal Y}
\newcommand{\calZ}{\mathcal Z}
\newcommand{\calX}{\mathcal X}
\newcommand{\calC}{\mathcal C}
\newcommand{\frakM}{\mathfrak M}

\makeatletter
\def\@settitle{\begin{center}%
  \baselineskip14\p@\relax
  \bfseries
  \uppercasenonmath\@title
  \@title
  \end{center}%
}
\def\@setauthors{%
  \begingroup
  \def\thanks{\protect\thanks@warning}%
  \trivlist
  \centering\footnotesize \@topsep30\p@\relax
  \advance\@topsep by -\baselineskip
  \item\relax
  \author@andify\authors
  \def\\{\protect\linebreak}%
  \MakeUppercase{\authors}%
  \ifx\@empty\contribs
  \else
    ,\penalty-3 \space \@setcontribs
    \@closetoccontribs
  \fi
  \endtrivlist
  \endgroup
}
\def\@setabstracta{%
  \ifvoid\abstractbox
  \else
    \skip@38\p@ \advance\skip@-\lastskip
    \advance\skip@-\baselineskip \vskip\skip@
    \box\abstractbox
    \prevdepth\z@
  \fi
}
\makeatother

\title{Positive Holomorphic Sectional Curvature on Rational Surfaces}

\author{Shiyu Zhang}
\address{University of Science and Technology of China, Hefei, 230026, China}
\email{shiyu123@mail.ustc.edu.cn}
\date{\today}
\subjclass[2020]{Primary 32Q15; Secondary 14J26, 14M25, 53C55}
\keywords{Positive holomorphic sectional curvature, K\"ahler metrics, rational surfaces, toric manifolds}

\begin{document}

\begin{abstract}

In 1975, Hitchin proved that any compact complex surface admitting a K\"ahler
metric with positive holomorphic sectional curvature $\HSC>0$ is rational.
Conversely, he constructed such metrics on all Hirzebruch surfaces
$\mathbb F_k$, as a first step toward characterizing rational surfaces by the
existence of a K\"ahler metric with suitable curvature positivity.

In this paper, we prove that every projective manifold $X$ obtained from a
projective toric manifold by a finite sequence of blow-ups at points admits a
K\"ahler metric with $\HSC>0$. This statement applies to all rational surfaces
and therefore completes Hitchin's result, thus resolving the surface
case of a problem of Yau listed in \emph{Open Problems in Geometry}.

The proof has two main ingredients. First, we prove that the toric
K\"ahler metric on a projective toric manifold arising from Delzant's construction has \(\HSC>0\). Second,
via a one-parameter degeneration, we construct, for any such
\(X\), a smooth projective family \(\pi:\mathcal X\to\C\) such that
\(\mathcal X_t\simeq X\) for \(t\ne0\), while \(\mathcal X_0\) is a
projective toric manifold.

\end{abstract}

\maketitle
\markboth{SHIYU ZHANG}{ POSITIVE HOLOMORPHIC SECTIONAL CURVATURE ON RATIONAL SURFACES}

\begingroup
\hypersetup{linkcolor=toclinkred}
\tableofcontents
\endgroup

\section{Introduction}

Let \((X,J,\w)\) be a K\"ahler manifold. For any nonzero tangent vector
\(V\in T_pX\), the holomorphic sectional curvature of the \(J\)-invariant
real two-plane \(\pi=\Span\{V,JV\}\) is defined by
\[
  \HSC(\pi)=\frac{R(V,JV,JV,V)}{\|V\|^4},
\]
where $R$ denotes the Riemannian curvature tensor of the metric $g(\cdot,\cdot)=\w(\cdot,J\cdot)$.

In 1975, Hitchin \cite{Hitchin1975} proved that a compact K\"ahler surface with positive holomorphic sectional curvature (abbreviated as $\HSC>0$) is rational, i.e., birationally equivalent to $\CP^2$. He then asked: ``\emph{With what sort of positivity of curvature can we attempt to characterize rational surfaces?}'' As a first step, he constructed K\"ahler metrics with $\HSC>0$ on all Hirzebruch surfaces $\mathbb F_{k}$.

A stronger curvature condition is positive holomorphic bisectional curvature \(\mathrm{HBSC}>0\). In foundational work, Mori and Siu-Yau \cite{Mori1979,SiuYau1980} independently established that any compact K\"ahler manifold with $\mathrm{HBSC}>0$ must be $\CP^n$. Motivated by these results and Hitchin's construction, Yau asked \emph{whether positivity of holomorphic sectional curvature can be used to characterize the rationality of projective manifolds} (cf. \cite{Yau1991}). Specifically, Yau posed the following question in his problem list \emph{Open Problems in Geometry}.

\begin{problem}[{\cite[Problem 67]{Yau1993}}]\label{prob:Yau}
  Let $X$ be a compact K\"ahler manifold with $\HSC>0$, is it
unirational? Is it projective? 

If a projective manifold is obtained by blowing
up a projective manifold with $\HSC>0$ along
a subvariety, does it still carry a metric with $\HSC>0$?
\end{problem}

\begin{remark}
  Blow-ups are among the basic operations in birational geometry. For instance,
  every smooth rational surface is obtained from \(\mathbb P^2\) or from a
  Hirzebruch surface \(\mathbb F_{k}\) by a finite sequence of blow-ups at
  points. In the surface case, the second part of
  Problem~\ref{prob:Yau} is equivalent to asking \emph{whether every rational
  surface admits a K\"ahler metric with $\HSC>0$}.
\end{remark}

Significant progress has been made on the first part of Problem~\ref{prob:Yau}. In complex dimension~$2$, this was showed by Hitchin. In higher dimensions, building on the existence of MRC fibrations~\cite{Cam92,KMM} and the uniruledness criterion~\cite{BDPP}, Heier-Wong~\cite{HeierWong2020} and Yang~\cite{Yang2018} proved that $\HSC>0$ implies rational connectedness in the projective and K\"ahler cases, respectively. Along these lines, Matsumura~\cite{Matsumura2020MRC,Matsumura2022} established structure theorems for projective manifolds with $\HSC\geq0$ in terms of MRC fibrations, which were subsequently extended to the K\"ahler case~\cite{ZZ2023,matsumura2025}; we refer the reader to Matsumura's survey~\cite{matsumurasurvey} for a systematic introduction. We also refer to~\cite{Ni2021,NiZheng2022,ZhangZhang2026} and the references therein for other interesting curvature conditions ensuring rational connectedness.

The main motivation for this paper stems from the second part of Problem~\ref{prob:Yau}, namely, the problem of finding K\"ahler metrics with \(\HSC>0\) on blow-ups of projective manifolds with \(\HSC>0\). Several extensions of Hitchin's construction have been obtained. Alvarez--Heier--Zheng \cite{AlvarezHeierZheng2018} proved that the projectivization of a holomorphic vector bundle over a compact K\"ahler manifold with \(\HSC>0\) again admits a K\"ahler metric with \(\HSC>0\). Furthermore, the class of K\"ahler metrics with \(\HSC>0\) on Hirzebruch manifolds
\[
\mathbb F_{n,k}
  :=
  \mathbb P
  \bigl(\mathcal O_{\mathbb P^n}
  \oplus \mathcal O_{\mathbb P^n}(k)\bigr)
\]
has been studied by Yang--Zheng \cite{YanZhe19}. However, even in complex dimension two, no examples with \(\HSC>0\) beyond \(\CP^2\) and the Hirzebruch surfaces constructed by Hitchin had previously appeared in the literature. For instance, prior to the present work, it remained unknown whether the blow-up of a Hirzebruch surface at one point admits a K\"ahler metric with \(\HSC>0\) (cf. \cite[Problem~3.8]{MatsumuraOpenProblems2022}).

\subsection{Main results}

The main result of this paper is the following theorem.

\begin{theorem}\label{theorem-blowup-HSC}
Let \(X\) be a projective manifold obtained from a projective toric manifold by
a finite sequence of blow-ups at points. Then \(X\) admits a K\"ahler metric with
\(\HSC>0\).
\end{theorem}

This provides a large and flexible supply of examples with $\HSC>0$,
in view of the abundance of smooth projective toric manifolds; see, for instance,
\cite{CLS11}. For example, projective spaces $\mathbb P^n$ and Hirzebruch
manifolds are smooth projective toric manifolds; see Example~\ref{ex:Hirzebruch}. Recall that every smooth rational surface is obtained from
$\mathbb P^2$ or from a Hirzebruch surface by a finite sequence of blow-ups at
points, Theorem~\ref{theorem-blowup-HSC} yields K\"ahler metrics with $\HSC>0$ on all
rational surfaces. Together with Hitchin's result, this settle the aforementioned questions of Hitchin and Yau for rational surfaces.

\begin{theorem}\label{theorem-rational-HSC}
A compact complex surface \(X\) is rational if and only if it admits a K\"ahler
metric with \(\HSC>0\).
\end{theorem}

In other words, Problem~\ref{prob:Yau} is completely resolved in complex
dimension \(2\). Furthermore, combining this with Matsumura's structure theorem
for \(\HSC\geq0\), we obtain the following classification.

\begin{corollary}\label{cor:nonnegative-hsc-surfaces}
Let \(X\) be a compact K\"ahler surface. Then the following are equivalent:
\begin{enumerate}
    \item \(X\) admits a K\"ahler metric \(\omega\) with
    \(
        \HSC_\omega \geq 0.
    \)
    \item \(X\) is one of the following:
    \begin{enumerate}
        \item a rational surface;
        \item a finite \'etale quotient of a two-dimensional complex torus;
        \item a projectively unitary flat \(\mathbb P^1\)-bundle over an elliptic curve.
    \end{enumerate}
\end{enumerate}
\end{corollary}

\subsection{The strategy}\label{subsec:strategy}
 For the reader's convenience, we briefly outline the proof of Theorem \ref{theorem-blowup-HSC} and explain the organization of this paper. Our proof relies heavily on both the algebraic and symplectic aspects of toric manifolds. Thus, in Section \ref{sec:projective-toric-manifolds} we include a minimal description of toric manifolds needed in this paper, which may be useful to readers unfamiliar with this topic. 

In Section \ref{sec:hsc-kahler-reduction}, we view a projective toric manifold $X$ as the underlying complex manifold of the K\"ahler reduction of some flat complex space $$(M_\Delta,J_\Delta,g_\Delta,\w_\Delta):=(\C^d,J_0,g_0,\w_0)//N$$ for some Delzant polytope $\Delta$ from Delzant's construction (cf.~Section~\ref{subsec:symplectic-side}). Through direct computations using Gauss's equation and O'Neill's formula, we prove that $g_\Delta$ in fact has $\HSC>0$. This approach is motivated by Kobayashi's alternative proof (see \cite[Theorem 7.6.38]{Kobayashi1987Book}), from the Marsden--Weinstein--Meyer symplectic reduction perspective \cite{Meyer73,MW74}, of the semi-positivity of holomorphic sectional curvature of the natural \(L^2\)-K\"ahler metric arising from Atiyah-Bott's construction \cite{AtiyahBott1983} on the moduli space of stable vector bundles over a compact Riemann surface \cite{NarasimhanSeshadri1965}, a result originally proved by Itoh \cite{Itoh1988}. The strict positivity in our setting follows from a straightforward compactness argument on the moment level set; the same argument applies to general compact K\"ahler reductions of flat complex spaces (see Theorem \ref{thm:flat-quotient-curvature}).

In Section \ref{sec:toric-deformation-ordered-cluster-blowups}, we regard a projective toric manifold as a projective manifold equipped with an action of a complex torus \(G=(\C^*)^m\) with an open dense orbit. In general, blowing up a projective toric manifold at a point that is not \(G\)-fixed may yield a non-toric variety; for instance, see the degree $5$ del Pezzo surface (Example \ref{ex:del-Pezzo-surface}). Consequently, Theorem \ref{thm:flat-quotient-curvature} does not directly apply in the general setting of Theorem \ref{theorem-blowup-HSC}. We address this issue as follows. Let \(X=\Bl_{\calC}Y\), where \(Y\) is a projective toric manifold and \(\calC\) is an ordered cluster of length \(k\) on \(Y\) (see Definition \ref{def:ordered-cluster}). Our strategy is to construct a smooth projective family
\(
\pi:\calX\longrightarrow \mathbb D
\)
over a disk such that
\begin{equation}\label{eq:smooth-family}
  X_t\simeq \Bl_{\calC}Y\quad (t\neq 0),
\qquad
X_0\simeq \Bl_{\calC_0}Y .
\end{equation}
The central fibre is thus a projective toric manifold. Combining the existence of a K\"ahler metric with $\HSC>0$ on $X_0$ (see Theorem \ref{thm:guillemin-positive}) with the Kodaira--Spencer local stability theorem for K\"ahler structures \cite{KodairaSpencer60}, we obtain a K\"ahler metric with \(\HSC>0\) on the nearby fibres, and hence on \(X\). To construct the above smooth projective family
\(\pi:\calX\rightarrow \mathbb D\), we proceed in two steps:
\begin{itemize}
  \item[(1)] \emph{The universal family of ordered cluster blow-ups.} In Section \ref{subsec:universal-family-ordered-clusters}, we review Kleiman's construction \cite{Kleiman81} for iterated blow-ups, which yields a smooth projective morphism
\(
\pi_k:\calY_k\longrightarrow \calY_{k-1}
\)
where \(\calY_{k-1}\) parametrizes ordered clusters of length \(k\) on \(Y\), and the fibre over the point \(c\in \calY_{k-1}\) corresponding to \(\calC\) is naturally isomorphic to \(\Bl_{\calC}Y\) (cf.~Proposition \ref{prop:kleiman-iteration}). These identifications may be regarded as higher-dimensional analogues of Harbourne's Proposition I.2 in \cite{Harbourne88}, where the case \(Y=\CP^2\) is considered. The torus action on \(Y\) lifts functorially to all \(\calY_r\), and the \(G\)-fixed points of \(\calY_{k-1}\) correspond exactly to toric ordered clusters (see Proposition \ref{prop:fixed-clusters}).
\item[(2)] \emph{One-parameter degeneration to a toric ordered cluster.} In Section \ref{subsec:deformation-central-toric-fibre}, we choose a generic one-parameter subgroup
\(
\lambda:\C^*\longrightarrow G
\) with $(\calY_{r-1})^\lambda=(\calY_{r-1})^G$ (cf.~Proposition \ref{prop:generic-one-parameter-calY}). For the point \(c\in \calY_{k-1}\), projectivity ensures that the orbit map
\[
\C^*\longrightarrow \calY_{k-1},
\qquad
t\longmapsto \lambda(t)c
\]
extends across \(0\). Its limit \(c_0\) is \(G\)-fixed, and therefore corresponds to a toric ordered cluster \(\calC_0\). Pulling back \(\pi_k\) along this extended orbit yields the desired smooth projective family $\pi:\calX\rightarrow\C$.
\end{itemize}
In summary, we prove Theorem~\ref{theorem-blowup-HSC} by establishing the following two statements:
\begin{enumerate}
  \item Theorem~\ref{thm:guillemin-positive}: the canonical toric K\"ahler metric on a symplectic toric manifold has \(\HSC>0\);

  \item Theorem~\ref{thm:toric-special-fibre}: there exists a smooth projective family \(\pi:\calX\to\C\) satisfying \eqref{eq:smooth-family}.
\end{enumerate}

\subsection{Further discussions}

We close the paper with some remarks on future directions.

\subsubsection{Classification of threefolds with $\HSC\geq0$}

According to Matsumura's structure theorem (cf. \cite[Theorem 1.1]{matsumura2025}), the main difficulty lies in classifying all rationally connected manifolds with $\HSC\geq0$. The classification of compact K\"ahler $3$-folds with \(\HSC\geq 0\) remains largely open.
\begin{enumerate}
  \item For compact K\"ahler manifolds with $\HSC>0$ in dimension at least \(3\), the existing results only establish rational connectedness. In dimension \(2\), rational connectedness is equivalent to rationality. In dimension at least \(3\), however, rationality and rational connectedness are no longer equivalent; they are related by the chain
  \begin{equation}\label{eq:rationality-line}
    \text{rational}
    \Longrightarrow
    \text{stably rational}
    \Longrightarrow
    \text{unirational}
    \Longrightarrow
    \text{rationally connected}.
  \end{equation}
  We note that all examples with \(\HSC>0\) produced by Theorem~\ref{theorem-blowup-HSC}---namely, projective manifolds obtained from projective toric manifolds by a finite sequence of blow-ups at points---are rational. Another known class of examples is due to Alvarez--Heier--Zheng \cite{AlvarezHeierZheng2018}, who proved that the projectivization of a vector bundle over a compact K\"ahler manifold with \(\HSC>0\) again admits a K\"ahler metric with \(\HSC>0\). Since the projectivization of a vector bundle over a rational manifold is rational, all known examples with \(\HSC>0\) arising from these two constructions are rational. Thus, the precise position of the condition \(\HSC>0\) in the hierarchy \eqref{eq:rationality-line} remains unclear\footnote{We remark that a recent paper \cite{LiZhangZhang2025} proves that a projective
manifold is rationally connected if and only if its tangent bundle is mean
curvature positive.  Moreover, \(\HSC>0\) implies mean curvature positivity
(cf. \cite[Proposition 3.6 and Remark 3.7]{LiZhangZhang2025}).}.

  \item Another difficulty concerns the construction of K\"ahler metrics with \(\HSC>0\) on blow-ups, namely the three-dimensional case of Problem \ref{prob:Yau}. In dimension at least \(3\), one must consider blow-ups along smooth curves, and our approach using Kleiman's parameter spaces for finite sequences of point blow-ups no longer applies.
\end{enumerate}

\subsubsection{Relationship between $\Ric>0$ and $\HSC>0$}

Recently, Brown \cite{Brown2024} proved that every uniruled surface admits a K\"ahler metric of positive scalar curvature $\mathrm{S}>0$. Together with Theorem \ref{theorem-rational-HSC} and the Calabi--Yau theorem \cite{Yau1978}, the correspondence between the algebraic classification of compact complex surfaces and the existence of K\"ahler metrics with natural curvature positivity is now well understood:
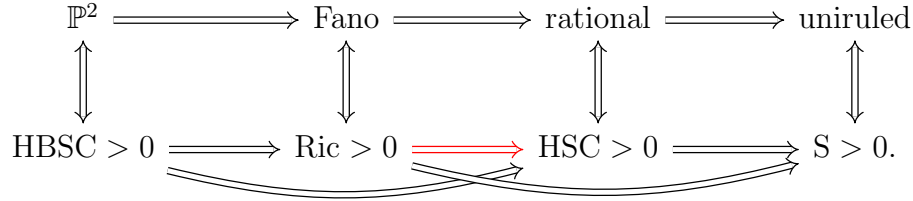
\[
\begin{tikzcd}[column sep=large,row sep=large]
  \mathbb P^2 \arrow[r, Rightarrow] \arrow[d, Leftrightarrow]
  & \mathrm{Fano} \arrow[r, Rightarrow] \arrow[d, Leftrightarrow]
  & \mathrm{rational} \arrow[r, Rightarrow] \arrow[d, Leftrightarrow]
  & \mathrm{uniruled} \arrow[d, Leftrightarrow] \\
  {\operatorname{HBSC}>0} \arrow[r, Rightarrow] \arrow[rr, Rightarrow, bend right=14]
  & {\Ric>0} \arrow[r, red, Rightarrow] \arrow[rr, Rightarrow, bend right=14]
  & {\HSC>0} \arrow[r, Rightarrow]
  & {\mathrm{S}>0}.
\end{tikzcd}
\]
Among these curvature positivity conditions, the black arrows hold for a single K\"ahler metric and follow directly from algebraic relations among the curvature tensors. The red arrow follows from the Calabi--Yau theorem and Theorem \ref{theorem-rational-HSC}. In the negative counterpart, the Wu--Yau theorem \cite{WuYau2016,TosattiYang2017} asserts ``$\HSC<0\Rightarrow\Ric<0$'' for compact K\"ahler manifolds. In higher dimensions, the relationship between $\Ric>0$ and $\HSC>0$ remains poorly understood. A classification of manifolds with $\HSC> 0$ may be helpful to this issue.

\begin{remark}
  The basic idea originated from the attempt to find a K\"ahler metric with $\HSC>0$ on the toy model $\mathrm{Bl}_{p_1,p_2}\CP^2$. Using computer-assisted numerical searches and Guillemin's formula \cite{Guillemin1994} for the canonical symplectic potential, the author was surprised to find that any canonical toric K\"ahler metric on $\mathrm{Bl}_{p_1,p_2}\CP^2$ appears to have $\HSC>0$.
\end{remark}

\medskip

\textbf{Acknowledgement.} The author would like to express his sincere gratitude to his Ph.D. supervisor, Xi Zhang, for constant encouragement and support. He would also like to thank Shin-ichi Matsumura for a private discussion at Tohoku University in 2024, which drew his attention to the remaining part of Problem \ref{prob:Yau}.

\section{Preliminaries on toric manifolds}
\addtocontents{toc}{\protect\setcounter{tocdepth}{1}}
\label{sec:projective-toric-manifolds}
In this section, we briefly review smooth projective toric manifolds from the
algebraic viewpoint and explain how such a manifold can be identified with the
underlying complex manifold of a symplectic toric manifold equipped with the
canonical toric K\"ahler structure arising from Delzant's construction, following \cite{Abreu2003,ACL03,Audin2004,Cannas08,CLS11}. The
purpose is to provide necessary facts used later in the proof.  For more comprehensive treatments, we refer to \cite{Weinstein77,Oda88,Fulton93,Cox03,Cannas08,McDuffSalamon2017} and the references therein. 

The bridge between the
algebraic and symplectic viewpoints is provided by Delzant polytopes.

\begin{definition}[Polytopes]
\label{def:delzant-polytope}
Let \(L\simeq\Z^n\) be a lattice and put \(L_\R=L\otimes_\Z\R\).  A
\emph{polytope} $\Delta$ in \(L_\R\) is the convex hull of finitely many points in
\(L_\R\).
\begin{itemize}
  \item $\Delta$ is called \emph{lattice} if
  all its vertices lie in \(L\).
  \item \(\Delta\) is called \emph{Delzant}
  if, at each vertex \(p\) of \(\Delta\), exactly \(n\) edges meet at \(p\), and
  these edges have primitive integral directions
  \(u_1,\ldots,u_n\in L\) which form a \(\Z\)-basis of \(L\).
\end{itemize}
\end{definition}

\subsection{Algebraic side}
\label{subsec:toric-algebraic-side}

\begin{definition}
A \emph{projective toric variety} is a projective variety $X$ equipped with an
action of a complex torus $(\C^*)^m$ having an open dense orbit $\mO$.
\end{definition}

Without loss of generality, we always assume that the action of $(\C^*)^m$ is effective.

\subsubsection{Examples}

Let us give some examples of projective toric varieties used later.

\begin{lemma}[Toric blow-ups]\label{lem:invariant-subvariety-blowup-toric}
Let $X$ be a projective toric variety with an action of a torus $G=(\C^*)^m$.  If
$Z\subsetneq X$ is a proper $G$-invariant closed subvariety, then $\Bl_Z X$ is a
projective toric variety with an action of $G$ such that the projection
\[
  \pi:\Bl_Z X\longrightarrow X
\]
is $G$-equivariant.
\end{lemma}

\begin{proof}
The centre $Z$ is $G$-invariant, so the functoriality of blow-ups lifts the effective
$G$-action to $\Bl_Z X$, and the projection is $G$-equivariant.  Let
$\mO\subset X$ be the open dense $G$-orbit.  Since $Z$ is closed,
$G$-invariant, and proper, one has $Z\cap\mO=\emptyset$; otherwise $Z$ would
contain the whole orbit $\mO$, hence its closure $X$.  Thus the blow-up map is
an isomorphism over $\mO$, so $\pi^{-1}(\mO)\simeq\mO$ is again an open dense
$G$-orbit.
\end{proof}

Let $Y$ be a projective toric manifold with an action of $(\C^*)^n$. Let $X$ be the blow-up of $Y$ at some point $p$. By Lemma \ref{lem:invariant-subvariety-blowup-toric}, $X$ is toric if $p$ is $(\C^*)^n$-fixed. Nevertheless, when $p$ is not $(\C^*)^n$-fixed, $X$ might not be toric.

\begin{example}[The degree \(5\) del Pezzo surface]\label{ex:del-Pezzo-surface}
  Consider the action of $(\C^*)^2$ on $\CP^2$ by
  $$[\w_1,\w_2]\cdot [z_1:z_2:z_3]=[\w_1z_1:\w_2z_2:z_3].$$
  The $(\C^*)^2$-open dense orbit is
  $$\bigl\{[z_1:z_2:z_3]:\ z_1z_2z_3\neq 0\bigr\}$$
  and the fixed points are the coordinate points:
  $$p_1=[1:0:0],\quad p_2=[0:1:0],\quad p_3=[0:0:1].$$
  By Lemma \ref{lem:invariant-subvariety-blowup-toric}, $X=\Bl_{p_1,p_2,p_3}\CP^2$ is toric with an action of $(\C^*)^2$. Let \(p_4\in \mathbb P^2\) be a general point, in particular not lying on any
coordinate line, and let \(\tilde p_4\in X\) be its inverse image. Then
\[
Y:=\Bl_{\tilde p_4}X\simeq \Bl_{p_1,p_2,p_3,p_4}\mathbb P^2
\]
is the degree \(5\) del Pezzo surface and its $\mathrm{Aut}(Y)$ is finite; see, for instance \cite[Theorem 8.5.6]{Dolgachev2012}. $Y$ cannot be toric; otherwise it would contain a dense open orbit of some torus and thus $(\C^*)^2\subset\mathrm{Aut}(Y)$.
\end{example}

\begin{example}[Projectivization]\label{ex:Hirzebruch}
    Let $E$ be the Whitney sum of holomorphic line bundles over $X$. If $X$ is a projective toric manifold, then $\CP(E)$ is also a projective toric manifold (see e.g. \cite[Section 7.3]{CLS11}). In particular, all projective spaces $\CP^n$ and Hirzebruch manifolds
    \(
  \mathbb F_{n,k}
\)
are projective toric manifolds.
\end{example}

Projective orbit closures are important models of projective toric varieties, we first review the definition of the character lattice.

\begin{definition}[Character lattice]\label{def:character-lattice}
For the torus \(G=(\C^*)^n\), we denote by
\[
  X^*(G):=\operatorname{Hom}_{\mathrm{alg}}(G,\C^*)
\]
the character lattice of \(G\).  By identifying \(X^*(G)\simeq\Z^n\), a
character \(\lambda=(\lambda_1,\ldots,\lambda_n)\in X^*(G)\) is the algebraic
homomorphism
\[
  \lambda:G\longrightarrow \C^*,
  \qquad
  \lambda(w)=w^\lambda:=w_1^{\lambda_1}\cdots w_n^{\lambda_n}.
\]
\end{definition}

\begin{example}[Projective orbit closures]
\label{ex:projective-orbit-closure-toric}
Let
\(
  A=\{\lambda^{(1)},\ldots,\lambda^{(k)}\}\subset X^*((\C^*)^n)
\)
be a finite set of characters.
These characters define an action of
\((\C^*)^n\) on \(\CP^{k-1}\) by
\[
  w\cdot[z_1:\cdots:z_k]
  =
  [w^{\lambda^{(1)}}z_1:\cdots:w^{\lambda^{(k)}}z_k].
\]
Set
\[
  X_A:
  =
  \overline{
  \bigl\{
  [w^{\lambda^{(1)}}:\cdots:w^{\lambda^{(k)}}]
  \mid w\in(\C^*)^n
  \bigr\}}
  \subset \CP^{k-1},
\]
i.e., \(X_A\) is the closure of the \((\C^*)^n\)-orbit through
\([1:\cdots:1]\).  Then \(X_A\) is a projective toric variety.
\end{example}

\subsubsection{Characterization of projective toric varieties}

By Sumihiro's equivariant embedding theorem \cite[Theorem 1]{Sumihiro74}, 
every normal projective toric variety embeds equivariantly into a projective 
space with a linear torus action. The standard theory of equivariantly 
projective toric varieties identifies this space with a projective orbit closure 
\(X_A\), corresponding to the lattice points \(A = P_A \cap \mathbb{Z}^n\) of a 
lattice polytope \(P_A\) \cite[Theorem II.3.1 and Section II.3.5]{ACL03}. 
Smoothness is then determined by the classical polytope criterion 
\cite[Theorems 2.4.3 and 3.1.19]{CLS11}; see also \cite[Sections 6.4--6.5]{ACL03}.

\begin{theorem}
\label{thm:acl-projective-orbit-closure}
Every normal projective toric variety \(X\) is equivariantly isomorphic to
a projective orbit closure \(X_A\), where
\[
  A=P_A\cap \mathbb Z^n
\]
is the set of lattice points of a lattice polytope \(P_A\). Furthermore, \(X_A\) is smooth if and only if \(P_A\) is Delzant.
\end{theorem}

\subsection{K\"ahler reduction}
\label{subsec:kahler-reduction-flat-complex-euclidean-space}

We first review K\"ahler reduction of the standard flat complex Euclidean
space, since the quotient construction of symplectic toric manifolds will be
used in precisely this form.

\medskip
\noindent\textbf{Reduction assumption 1.}
Let $(\C^n,J_0,g_0,\w_0)$ be the flat complex Euclidean space. 
Let $K$ be a compact Lie group with Lie algebra $\mathfrak k$, and let
\[
  \rho:K\longrightarrow U(n)
\]
be a unitary representation.
\medskip

Thus \(K\) acts on \(\C^n\) by
\(a\cdot z=\rho(a)z\).  In particular, \(J_0\), \(g_0\) and \(\w_0\) are
preserved under the action of \(K\).  For \(\xi\in\mathfrak k\), put
\[
  A_\xi=d\rho(\xi)\in\mathfrak u(n).
\]
We identify each tangent space $T_z\C^n$ with $\C^n$ by translation.  Then the
fundamental vector field generated by $\xi$ is
\begin{equation}
  \label{eq:fundamental-vector-field}
  X_\xi(z)
  =
  \left.\frac{d}{dt}\right|_{t=0}\rho(\exp(t\xi))z
  =
  A_\xi z.
\end{equation}
Recall that a moment map associated to $\rho:K\rightarrow U(n)$ is a smooth map $\mu:\C^n\to\mathfrak k^*$  satisfying
\[
  d\langle\mu,\xi\rangle=\iota_{X_\xi}\omega_0,
\]
for all $\xi\in\mathfrak k$. By a direct computation, one has:

\begin{statement}[The moment map]
\label{st:standard-quadratic-moment-map}
Let $(\mathfrak k^*)^K$ denote the subspace of fixed points of coadjoint action of $K$.  Define $\mu_0:\C^n\to\mathfrak k^*$ by
\begin{equation}
  \label{eq:standard-quadratic-moment-map}
  \langle \mu_0(z),\xi\rangle
  =\frac12\,\omega_0(A_\xi z,z).
\end{equation}
Then for any $c\in(\mathfrak k^*)^K$, $$\mu=\mu_0-c$$ is a moment map associated to $\rho:K\rightarrow U(n)$. Conversely, any moment map is of this form.
\end{statement}

The moment map \(\mu\) of the above form is equivariant with respect to the
\(K\)-action on \(\C^n\) and the coadjoint \(K\)-action on \(\mathfrak k^*\).
Hence, for any \(m\in\mu^{-1}(0)\) and any \(a\in K\), one has
\(\mu(a\cdot m)=0\).  In other words, the \(K\)-action on \(\C^n\) restricts to
a \(K\)-action on \(\mu^{-1}(0)\).

\medskip
\noindent\textbf{Reduction assumption 2.} Let $\mu$ be a moment map associated to $\rho:K\rightarrow U(n)$ and suppose that $K$ acts freely on $Z:=\mu^{-1}(0)$.

\begin{statement}[Symplectic reduction of Marsden--Weinstein \cite{MW74} or Meyer \cite{Meyer73}]
\label{st:mwm-zero-level-reduction}
Under the above assumptions, let \(\iota:Z\hookrightarrow\C^n\) be the inclusion map. Then:
\begin{itemize}
  \item $Z$ is a smooth submanifold of $\C^n$ of dimension \(2n-\dim K\), and
\(M=Z/K\)
is a smooth manifold of dimension \(2n-2\dim K\);
 \item The projection $\pi:Z\to M$ is a submersion;
 \item There is a symplectic form $\omega_M$ on $M$ such that
 \[
  \pi^*\omega_M=\iota^*\omega_0.
\]
\end{itemize}
\end{statement}

For $z\in Z$, the orbit $K\cdot z$ is an embedded submanifold with tangent space
at $z$ equal to
\[
  \Ver_z=\{A_\xi z:\xi\in\mathfrak k\}.
\]

\begin{statement}[The orthogonal splitting of the tangent bundle]
\label{st:orthogonal-splitting}
The normal directions to $Z\subset\C^n$ are generated by $J_0\Ver_z$.  With
\(
  \Hor_z=\Ver_z^{\perp_{g_0}}\cap T_zZ,
\)
one has the $g_0$-orthogonal splittings
\begin{equation}
  \label{eq:g0-orthogonal-splitting}
  T_z\C^n=J_0\Ver_z\oplus(\Ver_z\oplus\Hor_z),
  \qquad
  T_zZ=\Ver_z\oplus\Hor_z.
\end{equation}
\end{statement}
The naturally induced metric $g_M$ on $M$ is defined by horizontal lifts: if
$u,v\in T_{[z]}M$ and $\widetilde u,\widetilde v\in\Hor_z$ are the unique
vectors with $d\pi_z(\widetilde u)=u$ and $d\pi_z(\widetilde v)=v$, then
\[
  g_M(u,v)=g_0(\widetilde u,\widetilde v).
\]
The ambient complex structure $J_0$ preserves $\Hor_z$.  Hence $J_0$
descends to an almost complex structure $J_M$ on $M$ by
\[
  J_Mu=d\pi_z(J_0\widetilde u).
\]

\begin{statement}[K\"ahler reduction of Guillemin--Sternberg
\cite{GuilleminSternberg1982} (see also {\cite[Section 3]{HKLR87}})]
\label{st:kahler-reduction-guillemin-sternberg}
  Under the above assumptions, $(M,J_M,g_M,\w_M)$ is a K\"ahler manifold.
\end{statement}

\subsection{Symplectic side}
\label{subsec:symplectic-side}

\begin{definition}[Symplectic toric manifold]
A \emph{symplectic toric manifold} is a compact connected symplectic manifold
\((M,\omega)\) of dimension \(2n\), together with an effective Hamiltonian
action of the torus \(\mT^n\) and a moment map
\[
  \mu:M\longrightarrow(\R^n)^*.
\]
\end{definition}

Delzant's theorem states the
following one-to-one correspondence between equivalence classes:
\[
  \begin{array}{ccc}
    \{\text{symplectic toric manifolds}\}
    & \stackrel{1\text{--}1}{\longleftrightarrow} &
    \{\text{Delzant polytopes}\} \\
    (M^{2n},\omega,\mT^n,\mu)
    & \longmapsto &
    \mu(M).
  \end{array}
\]

We will describe the construction of toric manifolds from Delzant
polytopes, following \cite{Delzant88}, as an instance of the
K\"ahler reduction of flat complex Euclidean space described in
Section~\ref{subsec:kahler-reduction-flat-complex-euclidean-space}: the compact
group is \(N\subset \mT^d\), acting unitarily on \(\C^d\), and the corresponding
moment map is \(\phi_N\) below.

\subsubsection{Delzant's construction}
\label{subsubsec:delzant-construction}
  Let
\(\Delta\subset(\R^n)^*\) be a Delzant polytope, written as
\[
  \Delta
  =
  \{x\in(\R^n)^*\mid \langle x,v_i\rangle\le \lambda_i,\ 
  i=1,\ldots,d\},
\]
where \(v_i\in\Z^n\) are the primitive outward normals to the facets.  Define
\[
  \pi:\R^d\longrightarrow\R^n,
  \qquad
  \pi(e_i)=v_i.
\]
The Delzant condition implies that \(\pi(\Z^d)=\Z^n\), hence \(\pi\) induces
a surjective homomorphism
\[
  \Pi:\mT^d=\R^d/\Z^d\longrightarrow \mT^n=\R^n/\Z^n.
\]
Let \(N=\ker\Pi\), and let \(\mathfrak n=\ker\pi\) be its Lie algebra.  Thus
we have exact sequences
\[
  0\longrightarrow\mathfrak n
  \xrightarrow{\iota}\R^d
  \xrightarrow{\pi}\R^n
  \longrightarrow0
\]
and
\[
  0\longrightarrow(\R^n)^*
  \xrightarrow{\pi^*}(\R^d)^*
  \xrightarrow{\iota^*}\mathfrak n^*
  \longrightarrow0.
\]
For the standard action of \(\mT^d\) on \(\C^d\), the moment map is
\[
  \phi(z)
  =
  -\frac12\bigl(|z_1|^2,\ldots,|z_d|^2\bigr)+\lambda,
  \qquad
  \lambda=(\lambda_1,\ldots,\lambda_d)\in(\R^d)^*.
\]
The induced moment map for the \(N\)-action is
\[
  \phi_N=\iota^*\circ\phi:\C^d\longrightarrow\mathfrak n^*.
\]
\begin{proposition}[see e.g. {\cite[Page 241]{Cannas08}}]\label{prop:verify-assumption}
  Put \(Z=\phi_N^{-1}(0)\).  The
hypotheses needed in
Section~\ref{subsec:kahler-reduction-flat-complex-euclidean-space} are satisfied:
\begin{itemize}
  \item the level set \(Z=\phi_N^{-1}(0)\) is compact;
  \item \(N\) acts freely on \(Z\).
\end{itemize}
\end{proposition}
Therefore, by Statement~\ref{st:kahler-reduction-guillemin-sternberg}, the
K\"ahler reduction of the flat complex Euclidean space
\((\C^d,J_0,g_0,\w_0)\) by \(N\) is a compact K\"ahler manifold
\[
  (M_\Delta,J_\Delta,g_\Delta,\omega_\Delta):=(\C^d,J_0,g_0,\omega_0)//N,
  \qquad
  M_\Delta=Z/N.
\]
The induced \(\mT^n\simeq \mT^d/N\)-moment map has image \(\Delta\).

Conversely, for a symplectic toric manifold $(M,\w)$ with an action of $\mathbb T^n$ and a moment map $\mu:M\rightarrow (\R^n)^*$, the image of $\mu(M)$ is in fact a Delzant polytope $\Delta$. Delzant showed that there exists an $\mathbb{T}^n$-equivariant symplectomorphism
$$\psi:(M,\w,\mu)\xrightarrow{\sim} (M_\Delta,\w_\Delta,\mu_\Delta).$$
We call $(M_\Delta,J_\Delta,\w_\Delta)$ \emph{the canonical toric K\"ahler structure} of a given symplectic toric manifold $(M,\w)$.

\subsubsection{Symplectic vs Algebraic}
The main correspondences described above may be summarized as follows:
\[
  \begin{array}{ccc}
    \{\text{smooth projective toric varieties }X_A\}
     \longleftrightarrow
    \{\text{Delzant lattice polytopes }P_A\}
  \end{array}
\]
and, on the symplectic side,
\[
  \begin{array}{ccc}
    \{\text{Delzant polytopes }\Delta\}
     \longleftrightarrow
    \{\text{symplectic toric manifolds }(M_{\Delta},\w_{\Delta})\}.
  \end{array}
\]
Suppose that $P_A$ is a Delzant lattice polytope and
\(
  \iota_A:X_A\hookrightarrow\CP^{\ell-1}
\)
is the corresponding equivariant projective embedding. The restriction of the \((\C^*)^n\)-action to its real subgroup
\[
  \mT^n=\{(t_1,\ldots,t_n)\in(\C^*)^n \mid |t_i|=1 \text{ for all } i\}
\]
is effective, because the action of \((\C^*)^n\) on \(X_A\) was already
effective.  Put
\[
  \Delta_A=-\frac12 P_A.
\]
With our moment-map convention, the \(\mT^n\)-action on \(X_A\) is Hamiltonian
with moment map
\[
  \mu_A:X_A\longrightarrow(\R^n)^*,
  \qquad
  \mu_A([z_1:\cdots:z_k])
  =
  -\frac12\,
  \frac{\sum_{j=1}^k \lambda^{(j)} |z_j|^2}
       {\sum_{j=1}^k |z_j|^2}.
\]
Its image is \(\Delta_A\) (cf. {\cite[Lecture 6, Section 6.6]{ACL03}}).
Therefore, by
Delzant's theorem, there exists a \(\mT^n\)-equivariant symplectomorphism
    \[
      \psi_A:
      (X_A,\iota_A^*(\w_{\mathrm{FS}}),\mu_A)
      \xrightarrow{\ \sim\ }
      (M_{\Delta_A},\w_{\Delta_A},\mu_{\Delta_A}).
    \]
	The complex structure of $X_A$ induces a complex structure $J_A$ on the symplectic manifold $(M_{\Delta_A},\w_{\Delta_A})$. By \cite[Proposition A.1]{Abreu2003}, there exists an \(\mT^n\)-equivariant biholomorphism
	\[
	  \varphi_A:
	  (M_{\Delta_A},J_A,\mT^n)
	  \xrightarrow{\ \sim\ }
	  (M_{\Delta_A},J_{\Delta_A},\mT^n).
	\]

  \vspace{0.2cm}

  In summary, we have the following statement.
\begin{proposition}\label{prop:alg-sym}
    Every projective toric manifold is equivariantly biholomorphic to the underlying complex manifold of a symplectic toric manifold equipped with its canonical toric complex structure.
\end{proposition}

We remark that $\psi_A^*\varphi_A^*\w_{\Delta_A}$ and $\iota_A^*\w_{\mathrm{FS}}$ is different in general.

\medskip

\addtocontents{toc}{\protect\setcounter{tocdepth}{2}}
\section{Positive holomorphic sectional curvature on K\"ahler reduction}
\label{sec:hsc-kahler-reduction}

In this section we prove the following statement and then conclude that any projective toric manifold admits a K\"ahler metric with $\HSC>0$.

\begin{theorem}\label{theorem-reduction-hsc}\label{thm:flat-quotient-curvature}
Let $(M,g_M)$ be a K\"ahler reduction of the flat complex Euclidean space introduced in Section \ref{subsec:kahler-reduction-flat-complex-euclidean-space}. Then the quotient K\"ahler metric $g_M$ has 
\(
  \HSC_{g_M}\ge0.
\)
  If, in addition, the level set $Z=\mu^{-1}(0)$ is compact, then
\(
  \HSC_{g_M}>0.
\)
\end{theorem}

\begin{remark}
The semi-positivity of the holomorphic sectional curvature follows directly from Kobayashi's curvature formula for submersions of CR submanifolds \cite[Theorem 1.3]{Kobayashi1987} (see e.g. \cite[Section 3]{Manikandan23} for a discussion).
\end{remark} 

The key point of our argument is that compactness of the level set imposes
strict positivity, which relies on the following formula of the holomorphic sectional curvature in terms of $A_\xi$ and $\w_0$.

\begin{lemma}\label{lem:hsc-quadratic-term}
Let $u\in T_{[z]}M$ be any nonzero tangent vector at some point $[z]\in M$, and let $\widetilde u\in\Hor_z$ be its
horizontal lift.  Define the quadratic term $Q_{\widetilde u}\in\mathfrak k^*$ by
\begin{equation}
  \label{eq:quadratic-term}
  \langle Q_{\widetilde u},\xi\rangle
  =
  \frac12\,\omega_0(A_\xi\widetilde u,\widetilde u),
  \qquad \xi\in\mathfrak k.
\end{equation}
Then
\begin{equation}
  \label{eq:hsc-quadratic-term}
  \HSC_{g_M}(u)
  =
  \frac{16\,\|Q_{\widetilde u}\|_{B_z^{-1}}^2}{\|\widetilde{u}\|_{g_0}^4},
\end{equation}
where $B_z$ is the positive form on $\mathfrak{k}$ given by 
\[
  B_z(\eta,\xi)=g_0(A_\eta z,A_\xi z),
  \quad \forall\eta,\xi\in\mathfrak k.
\]
\end{lemma}

\medskip

The proof is a direct computation using Gauss equation for the embedding $Z\hookrightarrow \C^d$ and O'Neill's curvature formula for the quotient map $Z\rightarrow Z/N$. 

\begin{proof}[Proof of Lemma \ref{lem:hsc-quadratic-term}]

The freeness of the $K$-action on $Z$ implies that
\[
  \mathfrak k\longrightarrow T_zZ,\qquad
  \xi\longmapsto X_\xi(z)=A_\xi z
\]
is injective, and hence
$B_z$ is positive definite. For $\xi\in\mathfrak k$, write
\[
  N_\xi(w)=J_0A_\xi w.
\]
By \eqref{eq:g0-orthogonal-splitting}, the vectors $N_\xi(z)$ span the normal
space to $Z$ at $z$, and
\[
  g_0(N_\eta(z),N_\xi(z))=B_z(\eta,\xi).
\]

Let $\sff$ be the second fundamental form of $Z\subset (\C^n,g_0)$, and let
$\nabla$ denote the Levi-Civita connection of $g_0$. Since $g_0$ is flat, the Gauss equation for $Z\subset \C^n$ gives
\[
R^Z(\widetilde u,J_0\widetilde u,J_0\widetilde u,\widetilde u)
  =
  \langle \sff(\widetilde u,\widetilde u),\sff(J_0 \widetilde u,J_0 \widetilde u)\rangle-\| \sff(\widetilde u,J_0\widetilde u)\|^2.\]
Since \(N_\xi\) is the ambient linear vector field \(w\mapsto J_0A_\xi w\) and $\nabla$ is flat,
\[
  (\nabla_{\widetilde u}N_\xi)_z=J_0A_\xi\widetilde u
\]
and so
\[
  \langle \sff(\widetilde u,\widetilde u),N_\xi\rangle
  =-g_0(\nabla_{\widetilde{u}}N_\xi,\widetilde{u})=
  -g_0(J_0A_\xi\widetilde u,\widetilde u)
  =-\w_0(A_\xi\widetilde{u},\widetilde{u})=
  -2\langle Q_{\widetilde u},\xi\rangle.
\]
The same calculation gives
\[
  \sff(J_0\widetilde u,J_0\widetilde u)
  =
  \sff(\widetilde u,\widetilde u),
  \qquad
  \sff(\widetilde u,J_0\widetilde u)=0,
\]
because $A_\xi$ commutes with $J_0$ and is skew-adjoint with respect to $g_0$.
Therefore we have
\begin{equation}
  \label{eq:gauss-flat-ambient-curvature}
  R^Z(\widetilde u,J_0\widetilde u,J_0\widetilde u,\widetilde u)
  =
  \|\sff(\widetilde u,\widetilde u)\|^2=4\,\|Q_{\widetilde u}\|_{B_z^{-1}}^2.
\end{equation}

O'Neill's curvature formula \cite{ONeill66} gives
\[
  R^M(u,J_Mu,J_Mu,u)
  =
  R^Z(\widetilde u,J_0\widetilde u,J_0\widetilde u,\widetilde u)
  +3\|\mathcal A_{\widetilde u}(J_0\widetilde u)\|^2.
\] It remains to compute O'Neill's tensor $\mathcal A$, which is given by
\[
  \mathcal A_EF=(\nabla^Z_EF)^{\Ver},\quad\forall E,F\in\Hor_Z,
\]
 where $\nabla^Z$ is the Levi-Civita
connection of the induced metric on $Z$.  Let $E$ be a local horizontal
extension of $\widetilde u$.  Since $J_0$ preserves $\Hor$ and \eqref{eq:g0-orthogonal-splitting}, one has $$g_0(J_0E,X_\xi)=0$$ along $Z$.  Differentiating in the direction
of $E$ gives
\[
  0
  =
  g_0(\nabla^Z_E(J_0E),X_\xi)
  +g_0(J_0E,\nabla^Z_E X_\xi).
\]
At $z$, the first term is
$g_0(\mathcal A_{\widetilde u}(J_0\widetilde u),A_\xi z)$.  Since
$J_0\widetilde u\in \Hor_z\subset T_zZ$, we have
\begin{equation}
  \label{eq:oneill-tensor-quadratic-term}
  g_0\bigl(\mathcal A_{\widetilde u}(J_0\widetilde u),A_\xi z\bigr)
  =-g_0(J_0\widetilde u,\nabla_{\widetilde{u}}(A_\xi z))=-\w_0(\widetilde{u},A_\xi\widetilde u)=
  2\langle Q_{\widetilde u},\xi\rangle .
\end{equation}
It follows that
\begin{equation}
  \label{eq:oneill-quadratic-term-norm}
  \|\mathcal A_{\widetilde u}(J_0\widetilde u)\|^2
  =
  4\,\|Q_{\widetilde u}\|_{B_z^{-1}}^2.
\end{equation}
Combining this with \eqref{eq:gauss-flat-ambient-curvature} and
\eqref{eq:oneill-quadratic-term-norm}, and using
$\|u\|_{g_M}=\|\widetilde u\|_{g_0}$, proves
\eqref{eq:hsc-quadratic-term}.
\end{proof}

Now let us complete the proof of Theorem \ref{theorem-reduction-hsc}.

\begin{proof}[Proof of Theorem~\ref{thm:flat-quotient-curvature}]
Let $u\ne0$, and let $\widetilde u\in\Hor_z$ be its horizontal lift.  The
formula in Lemma~\ref{lem:hsc-quadratic-term} gives
\(\HSC_{g_M}(u)\ge0\).

Assume now that $Z$ is compact.  Since $\widetilde u\in T_zZ$, one has
$d\mu_z(\widetilde u)=0$.  The nonconstant part of the moment map is quadratic,
hence for real $t$,
\[
  \mu(z+t\widetilde u)
  =
  \mu(z)+t\,d\mu_z(\widetilde u)+t^2Q_{\widetilde u}
  =
  t^2Q_{\widetilde u}.
\]
If $Q_{\widetilde u}=0$, then the whole affine real line
$z+t\widetilde u$ lies in $Z=\mu^{-1}(0)$, contradicting the compactness assumption.  Therefore
$Q_{\widetilde u}\ne0$ and so
\eqref{eq:hsc-quadratic-term} gives \(\HSC_{g_M}(u)>0.\)
\end{proof}

As a direct consequence, we obtain the following statement.

\begin{theorem}\label{thm:guillemin-positive}
Let \((M_\Delta,J_\Delta,\omega_\Delta)\) be the canonical toric K\"ahler
manifold associated with a Delzant polytope \(\Delta\). Then \(\omega_\Delta\)
has \(\HSC>0\).
\end{theorem}

\begin{proof}
  Recall Delzant's construction and use the notations of Section \ref{subsubsec:delzant-construction}, the compact K\"ahler manifold
\(
  (M_{\Delta},J_{\Delta},\w_{\Delta})
\)
is the K\"ahler reduction of a standard flat complex space $(\C^d,J_0,g_0,\w_0)$, with an action of $N$ and the moment map $\phi_N$. By Proposition \ref{prop:verify-assumption}, the level set $Z:=\phi_N^{-1}(0)$ is compact.  Hence Theorem \ref{theorem-reduction-hsc} applies to $\w_\Delta$, and $\w_\Delta$ has $\HSC>0$.
\end{proof}

\begin{remark}
   Since every K\"ahler class on a projective toric manifold admits a canonical toric K\"ahler metric associated with a Delzant polytope, applying Theorem \ref{thm:guillemin-positive} to Hirzebruch manifolds recovers Yang-Zheng's result \cite[Theorem 1.3]{YanZhe19}.
\end{remark}

\medskip

\section{Toric deformation of ordered cluster blow-ups}
\label{sec:toric-deformation-ordered-cluster-blowups}

In this section, we aim to construct a smooth projective family satisfying \eqref{eq:smooth-family}.  We then combine this deformation
with Theorem \ref{thm:guillemin-positive} to complete the proof of Theorem \ref{theorem-blowup-HSC}.

\subsection{The universal family}
\label{subsec:universal-family-ordered-clusters}
This subsection is preparatory. We review Kleiman's construction
\cite{Kleiman81}  for iterated blow-ups, and clarify the lifted torus actions.

\begin{definition}\label{def:ordered-cluster}
Let $Y$ be a smooth projective manifold.  An \emph{ordered cluster of length $k$}
on $Y$ is a sequence
\[
  \calC=(p_1,\ldots,p_k)
\]
defined inductively: $p_1\in Y$, $p_2$ is a point on $\Bl_{p_1}Y$, and in
general $p_i$ is a point on the manifold $\Bl_{(p_1,\cdots,p_{i-1})}Y$ obtained after blowing up
$p_1,\ldots,p_{i-1}$. We write
$\Bl_{\calC}Y$ for the projective manifold obtained by blowing up $p_1,\cdots,p_{k}$ inductively.
\end{definition}

\subsubsection{Universal family}
 In the notation below, \(\calY_{r-1}\)
parametrizes ordered clusters of length \(r\), while
\(\pi_r:\calY_r\to\calY_{r-1}\) is the corresponding universal iterated
blow-up family.

\begin{proposition}\label{prop:kleiman-iteration}
Let \(Y\) be a smooth projective manifold.  Set
\(\calY_{-1}=\{\mathrm{pt}\}\), \(\calY_0=Y\), and let
\(\pi_0:\calY_0\to\calY_{-1}\) be the structure morphism.  Then, for every
\(r\ge1\), there is a smooth projective manifold \(\calY_r\) and a smooth
projective morphism (i.e., a holomorphic submersion with projective fibres)
\[
  \pi_r:\calY_r\longrightarrow\calY_{r-1}
\]
such that every point \(c\in\calY_{r-1}\) determines an ordered cluster
\(\calC=(p_1,\ldots,p_r)\) on \(Y\), together with a natural fibre
isomorphism
\[
  \Phi_c:(\calY_r)_c:=\pi_r^{-1}(c)
  \xrightarrow{\ \simeq\ }
  \Bl_{\calC}Y.
\]
The construction is inductive and one has $p_r:=\Phi_{c'}(c)$ for $c'=\pi_{r-1}(c)$. 
\end{proposition}

For \(Y=\CP^2\), analogous identifications appear in Harbourne's
Proposition I.2 \cite{Harbourne88}; see also \cite{Roe04,KPT11} for related
versions. The proof in the present setting is essentially the same, and we give
the construction and describe the fibres explicitly for the reader's
convenience.

\begin{proof}
Suppose that
\(
  \pi_{r-1}:\calY_{r-1}\to\calY_{r-2}
\)
has been constructed.  Let
\(
  \calZ_r=\calY_{r-1}\times_{\calY_{r-2}}\calY_{r-1},
\)
and let \(\Delta_r\subset\calZ_r\) be the relative diagonal.  Define
\[
  \calY_r=\Bl_{\Delta_r}\calZ_r,
\]
and let \(\pi_r:\calY_r\to\calY_{r-1}\) be the morphism induced by the first
projection \(\calZ_r\to\calY_{r-1}\).  Since fibre products and blow-ups along
closed centres preserve projectivity, each \(\pi_r\) is projective (cf.
\cite[II.7 and III.10]{Hartshorne77}).

A point \(c\in\calY_{k-1}\) first determines a chain of points as follows.  Put
\(c_{k-1}=c\), and let \(c_{-1}\) be the unique point of \(\calY_{-1}\).  If
\(k\ge2\), define
\[
  c_{j-1}=\pi_j(c_j)\in\calY_{j-1},
  \qquad j=k-1,\ldots,1.
\]
Set
\(
  \Phi_0:(\calY_0)_{c_{-1}}=Y\xrightarrow{\ \simeq\ }Y
\)
to be the identity and put \(p_1=\Phi_0(c_0)=c_0\).  We will prove by induction
on \(r=0,\ldots,k\) that \(\pi_r\) is smooth and that there is an isomorphism
\[
  \Phi_r:(\calY_r)_{c_{r-1}}
  \xrightarrow{\ \simeq\ }
  \Bl_{(p_1,\ldots,p_r)}Y.
\]
Once this isomorphism has been constructed, if \(r<k\), the next centre is
defined by
\[
  p_{r+1}:=\Phi_r(c_r)\in\Bl_{(p_1,\ldots,p_r)}Y.
\]

The case $r=0$ is clear. Assume that the construction has been completed up to \(r-1\).  Thus
\(\pi_{r-1}\) is smooth and we have an isomorphism
\[
  \Phi_{r-1}:(\calY_{r-1})_{c_{r-2}}
  \xrightarrow{\ \simeq\ }
  \Bl_{(p_1,\ldots,p_{r-1})}Y
\]
and the centre
\[
  p_r=\Phi_{r-1}(c_{r-1})
  \in\Bl_{(p_1,\ldots,p_{r-1})}Y.
\]
Let \(m=\dim Y\).
Since \(\pi_{r-1}\) is smooth of relative dimension \(m\), after choosing a
local analytic neighbourhood \(U\) of \(c_{r-2}\), choose fibre coordinates
\(x\) and \(y\) on the two factors of
\(\calZ_r=\calY_{r-1}\times_{\calY_{r-2}}\calY_{r-1}\):
\[
  V_x\simeq U\times\C^m_x,
  \qquad
  V_y\simeq U\times\C^m_y,
\]
with \(\pi_{r-1}\) given by projection to \(U\).  Thus, near
\((c_{r-1},c_{r-1})\in\Delta_r\),
\[
  \calZ_r|_U\simeq U\times\C^m_x\times\C^m_y,
  \qquad
  \Delta_r|_U=\{x=y\}.
\]
With respect to the first projection
\((b,x,y)\mapsto(b,x)\), set \(v=y-x\).  Then the pair is locally
\[
  (\calZ_r,\Delta_r)
  \simeq
  \bigl((U\times\C^m_x)\times\C^m_v,\,
        (U\times\C^m_x)\times\{0\}\bigr).
\]
Thus
\[
  \calY_r|_U
  =
  \Bl_{\Delta_r}\calZ_r|_U
  \simeq
  (U\times\C^m_x)\times\Bl_0\C^m_v,
\]
and \(\pi_r\) is the projection to \(U\times\C^m_x\), which proves smoothness.
Write \(c_{r-2}=b_0\) and \(c_{r-1}=(b_0,x_0)\) in the above coordinates. On the fibre \(b=b_0,\ x=x_0\), the coordinate \(v=y-x_0\)
identifies the centre \(v=0\) with the point \(y=x_0\), namely with
\(c_{r-1}\in(\calY_{r-1})_{c_{r-2}}\). Then restricting the preceding local model to the fibre over \(c_{r-1}\) gives the
natural isomorphism
\[
  \Psi_r:
  (\calY_r)_{c_{r-1}}\simeq (b_0,x_0)\times \Bl_0\C_v^m
  \simeq
  \Bl_{\{c_{r-1}\}}\bigl((\calY_{r-1})_{c_{r-2}}\bigr)
\]
Since \(p_r=\Phi_{r-1}(c_{r-1})\), the isomorphism \(\Phi_{r-1}\) induces an
isomorphism of blow-ups
\[
  \Bl_{\Phi_{r-1}}:
  \Bl_{\{c_{r-1}\}}\bigl((\calY_{r-1})_{c_{r-2}}\bigr)
  \xrightarrow{\ \simeq\ }
  \Bl_{\{p_r\}}\Bl_{(p_1,\ldots,p_{r-1})}Y
  =
  \Bl_{(p_1,\ldots,p_r)}Y
\]
and thus \(\Phi_r=\Bl_{\Phi_{r-1}}\circ\Psi_r\) gives the desired isomorphism.
This completes the induction step.
Taking \(r=k\) gives an ordered cluster
\(\calC=(p_1,\cdots,p_{k})\) and
\[
  \Phi_c:=\Phi_k:
  (\calY_k)_c=(\calY_k)_{c_{k-1}}
  \xrightarrow{\ \simeq\ }
  \Bl_{\calC}Y.
\]
The one-to-one correspondence is immediate from the recursive construction.
\end{proof}

\subsubsection{Lifts of group actions}
For
a \(G\)-variety \(X\), write \(X^G\) for its \(G\)-fixed point set.

\begin{lemma}\label{lem:G-action-lifts}
Let \(Y\) be a projective manifold with a $G$-action. The \(G\)-action on \(Y\) induces natural \(G\)-actions on every
\(\calY_r\), and each map
\[
  \pi_r:\calY_r\to\calY_{r-1}
\]
is \(G\)-equivariant.
\end{lemma}

\begin{proof}
The construction is functorial.  Inductively, suppose \(G\) acts on
\(\calY_{r-1}\) and \(\calY_{r-2}\), and that \(\pi_{r-1}\) is
\(G\)-equivariant.  Then
\(
  \calZ_r=\calY_{r-1}\times_{\calY_{r-2}}\calY_{r-1}
\)
carries the action by
\[
  g\cdot(y_1,y_2)=(g\cdot y_1,g\cdot y_2).
\]
This is well-defined because \(\pi_{r-1}\) is \(G\)-equivariant.  The relative
diagonal \(\Delta_r\) is invariant, and the action lifts to
\(
  \calY_r=\Bl_{\Delta_r}\calZ_r
\)
by the universal property of blowing up.  The first projection
\(\calZ_r\to\calY_{r-1}\) is \(G\)-equivariant, hence so is
\(\pi_r:\calY_r\to\calY_{r-1}\).
\end{proof}

\begin{definition}\label{def:toric-ordered-cluster}
An ordered cluster $(q_1,\ldots,q_k)$ on a projective toric manifold $Y$
is called \emph{toric} if
\[
  q_i\in\big(\Bl_{(q_1,\cdots,q_{i-1})}Y\big)^G,\quad \forall 1\leq i\leq k
\]
\end{definition}

\begin{proposition}\label{prop:fixed-clusters}
Let \(Y\) be a projective toric manifold with acting torus \(G=(\C^*)^m\). Under the correspondence given in Proposition
\ref{prop:kleiman-iteration}, the fixed point set \(\calY_{k-1}^G\) corresponds
exactly to the toric ordered clusters of length \(k\) on $Y$.  Moreover, if \(c\in\calY_{k-1}^G\) corresponds to the toric ordered cluster
\(\calC\), then the fibre isomorphism
\[
  \Phi_c:(\calY_k)_c\xrightarrow{\ \simeq\ }\Bl_{\calC}Y
\]
is \(G\)-equivariant.  In particular, it identifies \((\calY_k)_c\), with
its restricted \(G\)-action, with \(\Bl_{\calC}Y\) as a \(G\)-toric manifold.
\end{proposition}

\begin{proof}
We prove the fixed-point correspondence and the equivariance of the
fibre identifications simultaneously by induction on \(k\).

For \(k=1\), we have \(\calY_0=Y\).  Thus a point of
\(\calY_0^G=Y^G\) is exactly a toric ordered cluster of length one.  It
remains only to check the fibre identification.  By construction,
\[
  \calY_1=\Bl_{\Delta_Y}(Y\times Y),
\]
where \(G\) acts diagonally on \(Y\times Y\), and
\(\pi_1:\calY_1\to Y\) is induced by the first projection.  If
\(c\in Y^G\), then \(\{c\}\times Y\) is \(G\)-invariant and
\[
  \Delta_Y\cap(\{c\}\times Y)=\{(c,c)\}.
\]
Hence the fibre over \(c\) is naturally identified with
\[
  (\calY_1)_c
  \simeq
  \Bl_{(c,c)}(\{c\}\times Y)
  \simeq
  \Bl_cY.
\]
This identification is induced from the diagonal \(G\)-action, and is
therefore \(G\)-equivariant.

Assume that the statement has been proved for \(k-1\).  Let
\(b\in\calY_{k-1}^G\), and set
\[
  b'=\pi_{k-1}(b)\in\calY_{k-2}.
\]
Since \(\pi_{k-1}\) is \(G\)-equivariant, we have
\(b'\in\calY_{k-2}^G\).  By the induction hypothesis, \(b'\) corresponds to
a toric ordered cluster
\[
  \calC'=(q_1,\ldots,q_{k-1}),
\]
and the fibre identification
\[
  \Phi_{b'}:(\calY_{k-1})_{b'}
  \xrightarrow{\ \simeq\ }
  \Bl_{\calC'}Y
\]
is \(G\)-equivariant.  Since \(b\in(\calY_{k-1})_{b'}\) is \(G\)-fixed,
its image
\[
  q_k:=\Phi_{b'}(b)
\]
is a \(G\)-fixed point of \(\Bl_{\calC'}Y\).  Hence \(b\) corresponds,
via Proposition \ref{prop:kleiman-iteration}, to the toric ordered cluster
\[
  \calC=(q_1,\ldots,q_{k-1},q_k).
\]
Conversely, the same argument applied to \(\Phi_{c'}^{-1}\) gives the
reverse direction that $\calC$ is toric implies $c\in\calY_{k-1}^G$. 

It remains to show the $G$-equivariance of $\Phi_b$.  Let \(b\in\calY_{k-1}^G\) correspond to
\(\calC=(q_1,\ldots,q_k)\), and let
\(b'=\pi_{k-1}(b)\), corresponding to
\(\calC'=(q_1,\ldots,q_{k-1})\).  By construction,
\[
  \calY_k
  =
  \Bl_{\Delta_k}
  \bigl(\calY_{k-1}\times_{\calY_{k-2}}\calY_{k-1}\bigr).
\]
Restricting this relative blow-up to the fibre over \(b\) gives a natural
\(G\)-equivariant identification
\[
  (\calY_k)_b
  \simeq
  \Bl_b\bigl((\calY_{k-1})_{b'}\bigr),
\]
because the relative diagonal, the fibre over \(b\), and the centre \(b\)
are all \(G\)-invariant.  By the induction hypothesis, \(\Phi_{b'}\) is
\(G\)-equivariant and sends \(b\) to \(q_k\).  Therefore, by functoriality
of blow-ups, it induces a \(G\)-equivariant isomorphism
\[
  \Bl_b\bigl((\calY_{k-1})_{b'}\bigr)
  \xrightarrow{\ \simeq\ }
  \Bl_{q_k}\bigl(\Bl_{\calC'}Y\bigr)
  =
  \Bl_{\calC}Y.
\]
The resulting composition is precisely the fibre isomorphism \(\Phi_b\)
constructed in Proposition \ref{prop:kleiman-iteration}.  Hence
\(\Phi_b\) is \(G\)-equivariant.  This completes the induction.
\end{proof}

\subsection{Deformation with central toric fibre}
\label{subsec:deformation-central-toric-fibre}

In this subsection we prove the existence of the following deformation family, which is the second key ingredient in proving Theorem \ref{theorem-blowup-HSC}.

\begin{theorem}\label{thm:toric-special-fibre}
Let \(Y\) be a projective toric manifold with acting torus
\(G=(\C^*)^m\), and let \(\calC\) be an ordered cluster of length \(k\)
on \(Y\).  Then there exists a toric ordered cluster \(\calC_0\) of length
\(k\) on \(Y\) and a smooth projective family
\[
  \pi:\calX\longrightarrow \C
\]
such that, writing \(X_t:=\pi^{-1}(t)\), one has
\begin{equation}\label{eq:toric-deformation-fibres}
  X_t\simeq \Bl_{\calC}Y \quad (t\neq 0),
  \qquad
  X_0\simeq \Bl_{\calC_0}Y .
\end{equation}
In particular, the central fibre \(X_0\) is a smooth projective toric
manifold.
\end{theorem}

\subsubsection{A generic one-parameter subgroup over the universal family}\label{subsubsec:generic-one-parameter}

\begin{definition}
For the torus \(G=(\C^*)^m\), we denote by
\[
  X_*(G)=\operatorname{Hom}_{\mathrm{alg}}(\C^*,G)
\]
the cocharacter lattice of \(G\).  By identifying \(X_*(G)\simeq\Z^m\), a
cocharacter \(\lambda=(b_1,\ldots,b_m)\in X_*(G)\) is the one-parameter
subgroup
\[
  \lambda:\C^*\longrightarrow G,\qquad
  \lambda(t)=(t^{b_1},\ldots,t^{b_m}).
\]
\end{definition}

We need the following statement. It can be deduced from
Sumihiro's equivariant embedding theorem for torus actions
\cite{Sumihiro74}. In our situation, we give a self-contained proof using
only the projective orbit-closure description of projective toric manifolds
from Section~\ref{subsec:toric-algebraic-side}.

\begin{proposition}\label{prop:generic-one-parameter-calY}
Let \(Y\) be a smooth projective toric manifold with torus action by
\(G=(\C^*)^m\), and fix an integer \(k\ge0\). There exists a cocharacter
\(\lambda\in X_*(G)\) such that
\begin{equation}\label{eq:generic-one-parameter-calY-fixed}
  \calY_{r-1}^\lambda=\calY_{r-1}^G,\quad \forall\,0\le r\le k.
\end{equation}
\end{proposition}

It suffices to prove the following two lemmas.

\begin{lemma}\label{lem:generic-one-parameter}
Let \(\frakM_{\le k}\) be the collection of smooth projective
toric manifolds obtained from \(Y\) by at most \(k\) successive blow-ups
at \(G\)-fixed points. Then the collection
\(\frakM_{\le k}\) is finite.  Moreover, there exists a
cocharacter
\(
  \lambda\in X_*(G)
\)
such that
\[
  X^\lambda=X^G,\quad \forall X\in\frakM_{\le k}.
\]
\end{lemma}

\begin{lemma}\label{lem:lambda-fixed-calY}
Let \(\lambda\in X_*(G)\) be a cocharacter such that
\[
  X^\lambda=X^G,\quad \forall X\in\frakM_{\le k}.
\]
Then
\[
  \calY_{r-1}^\lambda=\calY_{r-1}^G,\quad \forall\,0\le r\le k.
\]
\end{lemma}

\begin{proof}[Proof of Lemma \ref{lem:generic-one-parameter}]
Let \(X\) be a smooth projective \(G\)-toric manifold.  By Theorem
\ref{thm:acl-projective-orbit-closure} and the notation of Example
\ref{ex:projective-orbit-closure-toric}, we may write
\[
  X\simeq X_A\subset \mathbb P^{N-1},
  \qquad
  A=\{\chi^{(1)},\ldots,\chi^{(N)}\}\subset X^*(G),
\]
where the \(\chi^{(i)}\) are distinct weights.  In this model every
\(G\)-fixed point has at most one nonzero homogeneous coordinate.  Hence
\(X^G\) is finite.  There are only finitely many possible toric blow-up
sequences of length at most \(k\): after each blow-up at a \(G\)-fixed point,
the resulting variety is again \(G\)-toric and has only finitely many
\(G\)-fixed points.  Hence \(\frakM_{\le k}\) is finite.

For each \(X\in\frakM_{\le k}\), choose such an equivariant orbit-closure
realization
\(
  X\simeq X_{A_X}\subset \mathbb P^{N_X-1},
\)
where
\[
  A_X=\mathbb Z^m\cap P_X
  =
  \{\chi_X^{(1)},\ldots,\chi_X^{(N_X)}\}\subset X^*(G).
\]
Let \(H\subset X_*(G)_\mathbb R\) be the finite union of hyperplanes given by
\[
  H=
  \bigcup_{X\in\frakM_{\le k}}
  \bigcup_{1\le i<j\le N_X}
  \left\{
    \lambda\in X_*(G)_\mathbb R
    \ \middle|\
    \bigl\langle
      \chi_X^{(i)}-\chi_X^{(j)},\lambda
    \bigr\rangle=0
  \right\}.
\]
Since the lattice \(X_*(G)\) is not contained in a finite union of proper hyperplanes, we may choose
\[
  \lambda\in X_*(G)\setminus H.
\]
Then, for every \(X\in\frakM_{\le k}\) and every
\(i\ne j\), one has
\[
  \bigl\langle \chi_X^{(i)},\lambda\bigr\rangle
  \ne
  \bigl\langle \chi_X^{(j)},\lambda\bigr\rangle .
\]

Let us show that \(X^\lambda=X^G\) for every \(X\in\frakM_{\le k}\).  The
inclusion \(X^G\subset X^\lambda\) is clear.
Conversely, take \(x\in X^\lambda\), and write it in the chosen
orbit-closure realization as
\[
  x=[z_1:\cdots:z_{N_X}]\in X_{A_X}\subset\mathbb P^{N_X-1}.
\]
The induced \(\C^*\)-action is
\[
  t\cdot [z_1:\cdots:z_{N_X}]
  =
  [
    t^{\langle\chi_X^{(1)},\lambda\rangle}z_1:
    \cdots:
    t^{\langle\chi_X^{(N_X)},\lambda\rangle}z_{N_X}
  ].
\]
If \(x\) is fixed by \(\lambda(\C^*)\), then all nonzero coordinates of
\(x\) must have the same exponent
\(\langle\chi_X^{(i)},\lambda\rangle\).  By the choice of \(\lambda\), this
is possible only when \(x\) has at most one nonzero coordinate.  Thus \(x\)
is a coordinate point of \(\mathbb P^{N_X-1}\), and hence it is fixed by the
whole torus \(G\).  Therefore \(X^\lambda\subset X^G\), and the equality
follows.
\end{proof}

\begin{proof}[Proof of Lemma \ref{lem:lambda-fixed-calY}]
Since \(\lambda(\C^*)\subset G\), the inclusion
\(\calY_{r-1}^G\subset \calY_{r-1}^\lambda\) is clear.  We prove the
reverse inclusion by induction on \(r\).  The cases \(r=0\) and \(r=1\) are trivial. Assume the assertion is known for \(r-1\) for some \(r\geq 2\). Let
\[
  c\in\calY_{r-1}^\lambda\ \text{ and }\
  c'=\pi_{r-1}(c)\in\calY_{r-2}.
\]
Since \(\pi_{r-1}\) is \(\lambda\)-equivariant, we have
\(c'\in\calY_{r-2}^\lambda\) and so \(c'\in\calY_{r-2}^G\) by the induction hypothesis. According to the correspondence of Proposition
\ref{prop:kleiman-iteration}, \(c'\) determines the preceding
ordered cluster \[\calC'=(p_1,\cdots,p_{r-1})\] of length \(r-1\) on \(Y\), and \(c\) determines \(\calC=(\calC',p_r)\) by putting \(p_r=\Phi_{c'}(c)\).  By Proposition
\ref{prop:fixed-clusters}, the ordered cluster $\calC'$ is toric and thus $\Bl_{\calC'}Y\in \frakM_{\le k}$. 

Note that \(c\) lies in this fibre and is fixed by \(\lambda\). The $G$-equivariant isomorphism given by Proposition \ref{prop:fixed-clusters}
\[
  \Phi_{c'}:(\calY_{r-1})_{c'}\rightarrow \Bl_{\calC'}Y
\]
shows that the point $p_r=\Phi_{c'}(c)$ of \(\Bl_{\calC'}Y\) is also fixed by $\lambda$. According to the assumption
\(
  (\Bl_{\calC'}Y)^\lambda=(\Bl_{\calC'}Y)^G,
\)
 $p_r$ is fixed by $G$ and so \(\calC=(\calC',p_r)\) is a toric ordered
cluster of length \(r\).  By Proposition \ref{prop:fixed-clusters},
\(
  c\in\calY_{r-1}^G.
\)
Therefore \(\calY_{r-1}^\lambda\subset\calY_{r-1}^G\), and the induction is
complete.
\end{proof}

The proposition follows immediately from Lemmas \ref{lem:generic-one-parameter}
and \ref{lem:lambda-fixed-calY}.

\subsubsection{Completion of the proof}

Now let us complete the proof of Theorem \ref{thm:toric-special-fibre}. Using the notations of Proposition \ref{prop:kleiman-iteration}, the cluster $\calC$ corresponds to a
point \(c\in\calY_{k-1}\) with \((\calY_k)_c\simeq \Bl_{\calC}Y\).  Choose
\(\lambda\in X_*(G)\) as in Proposition \ref{prop:generic-one-parameter-calY} and consider the orbit map
\[
  \C^*\to\calY_{k-1},\quad t\mapsto \lambda(t)c.
\]
Since \(\calY_{k-1}\) is projective, it is proper; hence the algebraic orbit map
\(t\mapsto\lambda(t)c\) extends across \(0\); analytically, this gives a
holomorphic map
\(
  \gamma:\C\longrightarrow \calY_{k-1}
\) with 
\[
\gamma(t)=\lambda(t)c,\quad \forall t\in \C^*.
\]
This is also clear after
choosing a projective embedding, since its homogeneous coordinates are
Laurent polynomials in \(t\), up to multiplication by a common power of \(t\).

Let
\(
  c_0=\gamma(0),
\)
then for any $s\in\C^*$,
\[
  \lambda(s)c_0
  =\lambda(s)\lim_{t\to0}\gamma(t)=
  \lambda(s)\lim_{t\to0}\lambda(t)c
  =
  \lim_{t\to0}\lambda(st)c
  =
  c_0.
\]
Thus \(c_0\in\calY_{k-1}^\lambda\).  By \eqref{eq:generic-one-parameter-calY-fixed},
\(c_0\in\calY_{k-1}^G\), and by Proposition \ref{prop:fixed-clusters} it is a toric ordered
cluster of length $k$.
Consider the fibre product
\[
  \calX
  =
  \calY_k\times_{\calY_{k-1}}\C
  =
  \{(y,t)\in \calY_k\times\C \mid \pi_k(y)=\gamma(t)\},
\]
and let
\[
  \pi:\calX\longrightarrow \C
\]
be the projection to the second factor.  Since
\(\pi_k:\calY_k\to\calY_{k-1}\) is smooth and projective by Proposition
\ref{prop:kleiman-iteration}, the base change
\(\pi:\calX\to\C\) is again a smooth projective family.

Let \(\calC_0\) be the toric ordered cluster corresponding to \(c_0\). It remains to prove the fibre identifications in \eqref{eq:toric-deformation-fibres}. For \(t\neq0\), \(\gamma(t)=\lambda(t)c\).  Since the correspondence in
Proposition \ref{prop:kleiman-iteration} is compatible with the \(G\)-action,
the point \(\lambda(t)c\) corresponds to the transported cluster
\(\lambda(t)\calC\).  Therefore
\[
  X_t=(\calY_k)_{\lambda(t)c}
      \simeq \Bl_{\lambda(t)\calC}Y
      \simeq \Bl_{\calC}Y,
\]
where the last isomorphism follows from
\(\lambda(t)\in G\subset\Aut(Y)\) and functoriality of blow-ups.
For \(t=0\), Proposition \ref{prop:kleiman-iteration} gives
\[
X_0=(\calY_k)_{c_0}\simeq \Bl_{\calC_0}Y.
\]
Since \(c_0\in (\calY_{k-1})^G\), Proposition \ref{prop:fixed-clusters}
shows that \(\calC_0\) is a toric ordered cluster. Hence \(X_0\) is a smooth
projective toric manifold by Lemma \ref{lem:invariant-subvariety-blowup-toric}. This proves
Theorem \ref{thm:toric-special-fibre}.

\subsection{Proof of main results}
We end the paper by proving Theorem \ref{theorem-blowup-HSC}, Theorem \ref{theorem-rational-HSC} and Corollary \ref{cor:nonnegative-hsc-surfaces} building on the preceding discussion.

\subsubsection{The existence of K\"ahler metrics with $\HSC>0$ on rational surfaces}
Recall the following statements:
\begin{itemize}
  \item $\CP^2$ and all Hirzebruch surfaces $\mathbb F_{k}$ are projective toric surfaces (see Example \ref{ex:Hirzebruch});
  \item according to the classical classification of compact complex surfaces (cf. the textbook \cite[Chapter III, Theorem 4.5; Chapter VI, Proposition 3.3]{BHPV04}), every rational surface is obtained from $\CP^2$ or from a Hirzebruch surface $\mathbb F_k$ by a finite sequence of blow-ups at points.
\end{itemize}
Thus, to prove Theorem \ref{theorem-rational-HSC}, it suffices to prove Theorem \ref{theorem-blowup-HSC}.

\begin{proof}[Proof of Theorem \ref{theorem-blowup-HSC}]
  We may assume that $X=\Bl_\calC Y$ for some projective toric manifold $Y$ and some ordered cluster of length $k$. Choose a smooth projective family $\pi:\calX\rightarrow \C$ given in Theorem \ref{thm:toric-special-fibre}; then we have
\begin{equation}\label{eq:main-proof-fibre-identifications}
  X_t\simeq \Bl_{\calC}Y \quad (t\neq 0),
  \qquad
  X_0\simeq \Bl_{\calC_0}Y
\end{equation}
for some toric ordered cluster $\calC_0$ of length $k$. In particular \(X_0\) is a projective toric manifold by Lemma \ref{lem:invariant-subvariety-blowup-toric}. Recall that  $X_0$ is biholomorphic to the underlying complex manifold of the canonical toric K\"ahler manifold $(M_\Delta,J_\Delta,\w_\Delta)$ obtained from Delzant's construction (cf. Proposition \ref{prop:alg-sym}). Theorem \ref{thm:guillemin-positive} implies that $\w_\Delta$ has $\HSC>0$ and thus gives a K\"ahler metric $\w$ with $\HSC>0$ on $X_0$. According to the classical Kodaira-Spencer local stability of K\"ahler structures \cite[Theorem 15]{KodairaSpencer60}: for a sufficiently small disk $\mathbb D$, the fibre $X_t$ over any $t\in\mathbb D$ admits a K\"ahler metric $\w_t$, which depends smoothly on $t$ and coincides for $t=0$ with $\w$. Since the condition \(\HSC>0\) is open in the \(C^2\)-topology of K\"ahler metrics, after shrinking \(\mathbb D\) if necessary, one has
\[\HSC_{\w_t}>0,\quad \forall t\in\mathbb D,\]
and we conclude that $\Bl_{\calC}Y$ admits a K\"ahler metric with $\HSC>0$ by \eqref{eq:main-proof-fibre-identifications}. 
\end{proof}

\subsubsection{Classification of surfaces with $\HSC\geq0$}\label{subse:classification-surface}
From the viewpoint of algebraic geometry, it is an interesting problem to classify all varieties of dimension two or three with $\HSC\geq0$, which was asked in \cite[Problem 5.3]{Matsumura2022}. 

Combining Matsumura's structure theorem obtained in \cite{Matsumura2022,matsumura2025} for the detailed statement) and Theorem \ref{theorem-rational-HSC}, we obtain a complete classification of compact K\"ahler surfaces with $\HSC\geq 0$. The proof is standard, but we include it for the sake of completeness.

\begin{proof}[Proof of Corollary \ref{cor:nonnegative-hsc-surfaces}]
  
We first prove \((1)\Rightarrow(2)\).  Let $(X,g)$ be a compact K\"ahler surface with $\HSC\geq 0$. For simplicity, we adopt the notation of \cite[Theorem 1.1]{matsumura2025}. Since \(\dim X=2\), there are three cases by Matsumura's structure theorem.  If \(\dim F=2\), then \(Y\) is a
point and \(X=F\) is a rationally connected projective surface.  Hence \(X\) is
rational by the Enriques--Kodaira classification
(see e.g. \cite[Chapter~VI]{BHPV04}).  If \(\dim F=0\), then \(X=Y\), and \(X\) is a
finite \'etale quotient of a two-dimensional complex torus.

It remains to consider the case \(\dim F=1\).  Then \(F\simeq\CP^1\), while
\(Y\) is a one-dimensional finite \'etale quotient of a compact complex torus,
hence an elliptic curve \(E\).  Since the fibration is locally constant, it is
given by
\(
  X\simeq (\C\times\CP^1)/\Lambda,
\)
where \(\Lambda=\pi_1(E)\) acts by
\(
  \lambda\cdot(z,p)=(z+\lambda,\rho(\lambda)p)
\)
for a representation
\[
  \rho:\Lambda\longrightarrow \Aut(\CP^1)=\mathrm{PGL}_2(\C).
\]
The isometric splitting in Matsumura's theorem implies that the fibre
monodromy acts by holomorphic isometries of \((\CP^1,g_F)\).  Thus
\(
  \rho(\Lambda)\subset \Aut(\CP^1)\cap \operatorname{Isom}(\CP^1,g_F),
\)
which implies that the closure \(\overline{\rho(\Lambda)}\) is compact. By Cartan's theorem on maximal
compact subgroups of semisimple Lie groups \cite[Chapter VII, Section
2]{Knapp2002}, every compact subgroup of \(\mathrm{PGL}_2(\C)\) is conjugate
into \(\operatorname{PU}(2)\).  Hence, after conjugating inside
\(\mathrm{PGL}_2(\C)\), we may assume that
\(
  \rho(\Lambda)\subset \operatorname{PU}(2).
\)
Therefore \(X\) is a projectively unitary flat \(\CP^1\)-bundle over an
elliptic curve.

Conversely, assume that \(X\) is one of the three surfaces in \((2)\).  If
\(X\) is rational, then Theorem~\ref{theorem-blowup-HSC}
gives a K\"ahler metric on \(X\) with \(\HSC>0\). If \(X\) is a finite \'etale quotient of a two-dimensional complex torus, then
\(X\) admits a flat K\"ahler metric by the complex version of the classical Bieberbach theorem; see, for instance
\cite[Lemma~1.3]{Rogov2022}. Finally, suppose that \(X\) is a projectively unitary flat \(\CP^1\)-bundle
over an elliptic curve \(E=\C/\Lambda\).  Then
\(
  X\simeq(\C\times\CP^1)/\Lambda
\)
for some representation \(\rho:\Lambda\to\operatorname{PU}(2)\).  The product
K\"ahler metric
\(
  \widetilde\w
  =
  \w_{\mathrm{flat},\C}+\w_{\mathrm{FS}}
\)
on \(\C\times\CP^1\) is \(\Lambda\)-invariant, and hence descends to a
K\"ahler metric \(\w_X\) on \(X\).  Since the product metric has no mixed
curvature terms, for every nonzero vector
\(v=(v_{\C},v_{\CP^1})\) one has
\(
  \HSC_{\widetilde\w}(v)
  \geq 0.
\)
Thus \(\w_X\) has semi-positive holomorphic sectional curvature.  This proves
\((2)\Rightarrow(1)\), and the proof is complete.
\end{proof}

\end{document}